\documentclass[12pt,reqno]{amsart}

\usepackage{amssymb,amsmath}
 \input{xy}
\xyoption{all}

\numberwithin{equation}{section}

\newtheorem{thm}{Theorem}[section]
\newtheorem{pro}[thm]{Proposition}
\newtheorem{lemma}[thm]{Lemma}
\newtheorem{cor}[thm]{Corollary}
\newtheorem{question}[thm]{Question}

\theoremstyle{definition}
\newtheorem{definition}[thm]{Definition}
\newtheorem{remark}[thm]{Remark}

\DeclareMathOperator*{\Vect}{Vect}
\DeclareMathOperator*{\PVect}{PVect}
\DeclareMathOperator*{\BL}{BL}
\newcommand{\Aut}{{\rm Aut}}
\newcommand{\Hom}{{\rm Hom}}
\newcommand{\scrK}{\mathcal{K}}
\newcommand{\scrp}{\mathfrak{p}} 
\newcommand{\scrq}{\mathfrak{q}} 
\newcommand{\scrE}{\mathcal{E}}

\newcommand{\cO}{\mathcal{O}}
\newcommand{\scrV}{\mathcal{V}}

\begin{document}

\title{The tame Nori fundamental group} 

\author[I. Biswas]{Indranil Biswas}

\address{Department of Mathematics, Shiv Nadar University, NH91, Tehsil
Dadri, Greater Noida, Uttar Pradesh 201314, India}

\email{indranil.biswas@snu.edu.in, indranil29@gmail.com}

\author[M. Kumar]{Manish Kumar}

\address{Statistics and Mathematics Unit, Indian Statistical Institute,
Bangalore 560059, India}

\email{manish@isibang.ac.in}

\author[A.J. Parameswaran]{A. J. Parameswaran}

\address{Kerala School of Mathematics, Kunnamangalam PO, Kozhikode, Kerala, 673571, India}

\email{param@ksom.res.in}



\begin{abstract}
 We introduce three notion of tameness of the Nori fundamental group scheme for a normal quasiprojective
variety $X$ over an algebraically closed field. It is proved that these three notions agree if $X$ admits a smooth
completion with strict normal crossing divisor as the complement. We also prove a Lefschetz
type restriction theorem for the tame Nori fundamental group scheme for such an $X$.
\end{abstract}

\maketitle

\section{Introduction}

Let $X$ be a variety defined over an algebraically closed field $k$, and fix a closed point $x$ of $X$.
Let $\pi^N(X,\,x)(k)$ denote the $k$-rational points of the Nori fundamental group scheme 
defined by Nori using the finite group scheme torsors over $X$ \cite{No}.
There is a natural surjective map $\pi^N(X,\,x)(k)\,\longrightarrow \,\pi_1^{et}(X,\,x)$
to the \'etale fundamental group, because $\pi_1^{et}(X,\,x)$ corresponds to the torsors on $X$ for (reduced)
finite groups.
On the \'etale side, when $X$ is quasi-projective normal curve one defines 
$\pi_1^{et,t}(X,\,x)$ as the quotient of $\pi_1^{et}(X,\,x)$ which represents \'etale covers of $X$ which are tamely ramified at the 
boundary points. The notion of tameness becomes a little more subtle for covers of higher dimensional normal varieties (see 
\cite{KS-tame}).

In this article we provide three different definitions of the tame Nori fundamental group scheme (see Proposition \ref{2.9} and
Proposition \ref{2.11}) 
and discuss relationship between them. When $X$ is projective then $\pi^N(X,\,x)$ has a Tannakian description. In \cite{Nori-pi1} we defined
Nori fundamental group of projective formal orbifolds as Tannaka dual of a Tannakian category of certain equivariant bundles. Moreover, using
this notion, we also
defined a group scheme $\pi^n(X^o,\,x)$ when $X^o$ is a quasi-projective normal variety. It was shown in \cite{Nori-pi1} that there is a surjection $\pi^N(X^o,\,x)\,\longrightarrow\,
\pi^n(X^o,\,x)$; however, it is not clear if this is an isomorphism. 

Two of the three versions of the tame Nori fundamental group 
schemes introduced here are quotients of $\pi^n(X^o,\,x)$ and all the versions agree if $X^o$ admits a smooth completion $X$ with $D=X\,\setminus\, X^o$ a 
strict normal crossing divisor (Proposition \ref{3.1} and Theorem \ref{orb-par-corr}).

In the above set-up, Lefschetz theorem holds for the tame 
fundamental group \cite{EK}. We partially extend it to the tame Nori fundamental group scheme. More precisely, in Theorem \ref{4.5} we show 
that if $Z$ is a general hyperplane section of $X$ high enough degree, then the natural
homomorphism $\pi^{N,t}(Z\,\setminus\, Z\cap D)\,\longrightarrow \,\pi^{N,t}(X^o)$ is surjective if the dimension of $Z$ is at least one.

\subsection*{Notation and Preliminaries}
We recall the notation for branch data and formal orbifold from \cite[Section 3]{formal.orbifolds} and \cite[Section 2]{Nori-pi1}. 

By a variety we mean an integral finite type scheme over a field $k$. Let $X$ be a normal variety over a perfect 
field $k$. For a point $x\,\in\, X$ of codimension at least one, let $\scrK_{X,x}$ be the fraction field of 
$\widehat{\cO}_{X,x}$. For an open affine neighborhood $U\,\subset\, X$ of $x$, let $\scrK_X^x(U)$ be the field of 
fractions of $\widehat{\cO_X(U)}^x$ where $\widehat{\cO_X(U)}^x$ is the completion of $\cO_X(U)$ along $x$. A 
branch data $P$ on $X$ assigns for every pair $(x,\,U)$ a finite Galois extension $P(x,U)$ of $\scrK_X^x(U)$ 
satisfying certain compatibility conditions (see \cite[Definition 3.3]{formal.orbifolds} or \cite[Section 
2]{Nori-pi1}. Let $P(x)\,=\,P(x,U)\scrK_{X,x}$ be the compositum. The branch locus $\BL(P)$ of $P$ consist of points 
$x\,\in\, X$ such that $\widehat{\cO}_{X,x}$ is branched in $P(x)$. One defines a morphism of formal orbifolds and 
$f\,:\,(Y,\,Q)\,\longrightarrow\, (X,\,P)$ is called \'etale if the field extension $Q(y)/P(f(x))$ is unramified for all 
$y\,\in \,Y$. A branch data in which every field extension is trivial is called the trivial branch data and is 
denoted by $O$. A branch data $P$ on $X$ is called geometric if there exists a Galois \'etale cover of formal 
orbifolds $f\,:\,(Y,\,O)\,\longrightarrow\, (X,\,P)$. In this case we also call $(X,\,P)$ to be a geometric
formal orbifold.

\section{Various notions of tame Nori fundamental group}

 Let $X^o$ be a regular quasiprojective variety over a perfect field $k$ of characteristic $p$. Let $X$ be a normal projective variety containing $X^o$ as an
open subset. Recall the following 
definitions of tameness from \cite{KS-tame}. An \'etale covering $Y^o\,\longrightarrow\,
X^o$ is said to 
have \emph{curve-tame} ramification at the boundary if the normalization of the pullback $C^o\times_{X^o} Y^o
\,\longrightarrow\, C^o$ is tamely ramified outside $C^o$ for every morphism $C^o\,\longrightarrow\, 
X^o$ with $C^o$ a regular quasi-projective curve. Note that an \'etale covering $D^o\longrightarrow C^o$ of quasi-projective 
regular curves is said to be tamely ramified outside $C^o$ if the induced morphism $D\,\longrightarrow\, 
C$ is tamely ramified at the points over $C\,\setminus\, C^o$ where $C$ and $D$ are the regular 
projective compactification of $C^o$ and $D^o$ respectively. Let $f:Y\,\longrightarrow\, X$ be the 
normalization of $X$ in the function field $k(Y)$ of $Y$ and further assume that $k(Y)/k(X)$ is a Galois extension. The covering $f$ is said to be 
\emph{numerically tame} if for all $y\in Y\,\setminus\, Y^o$, the inertia group $I_y$ is of 
order prime to $p$.

There are multiple ways to think about the tame Nori fundamental group scheme of $X^o$.

\subsection{Two definitions using formal orbifolds}

Let $(X,\,P)$ be a normal projective formal orbifold over $k$ such that $X^o$ is an open 
subset of $X$ and the branch locus of $P$ lies outside $X^o$.  In \cite{Nori-pi1} Nori 
fundamental group of such a pair $(X,\,P)$ was defined using the essentially finite Galois 
equivariant bundles of an appropriate cover of $(X,\,P)$. Recall from \cite[Definition 
9.5]{formal.orbifolds} that the branch data $P$ is said to be \emph{numerically tame} if 
for every point $x \,\in\, X$ of codimension at least one the extension $P(x)/\scrK_{X,x}$ is 
tamely ramified. Also, $P$ is called curve-tame at
a closed point $x\,\in\, X$ if the following holds:

For any coheight one prime ideal $\scrp$ in $\widehat{ \cO_{X,x}}$ (i.e., $\dim(\widehat{\cO_{X,x}}/\scrp)=1$) 
in the branch locus $\BL(P)$, and for any coheight one prime ideal $\scrq$ in the integral closure 
$\cO_{P(x)}$ of $\widehat{ \cO_{X,x}}$ in $P(x)$ lying above $\scrp$, let $R$ and $S$ be 
the normalizations of $\widehat{ \cO_{X,x}}/\scrp$ and $\cO_{P(x)}/\scrq$ respectively. 
Then $P$ is called curve-tame at 
$x\,\in\, X$ if $S/R$ is at most a tamely ramified extension.

The branch data $P$ is called \emph{curve-tame} if it is curve-tame for all closed points $x\,
\in\, X$. A branch data $P$ on $X$ is called {\it geometric} if $P\,=\,B_f$ for some finite cover $f\,:\,Y\,\longrightarrow\, X$.

A branch data which is both geometric and curve-tame will be called a geometric curve-tame
branch data. Recall the following result of \cite{formal.orbifolds} which characterizes
geometric curve-tame branch data.

\begin{pro}[{\cite[Proposition 9.6]{formal.orbifolds}}]
Let $X$ be a proper normal variety over $k$ and $f\,:\,Y\,\longrightarrow\, X$ a Galois
covering which is \'etale over a nonempty open subset $X^o$ of the regular locus of
$X$. The branch data $B_f$ is curve-tame if and only if $f$ is curve-tame.  
\end{pro}

Let $x\,\in\, X^o$ be a generic geometric point of $X^o$, so it is given by a morphism $${\rm 
Spec}(K)\,\longrightarrow\, X^o,$$ where $K$ is the separable closure of $k(X^o)$.  We define
$$\pi_1^{et,nt}(X^o)\,:=\,\pi_1^{et,nt}(X^o,\,x)$$ where $\pi_1^{et,nt}(X^o,\,x)$ is the quotient of $\pi_1^{et}(X^o,\, x)$ obtained by taking the inverse limit of the Galois groups of the Galois
coverings of $X$ which are numerically tame along $X\setminus X^o$. Note that $\pi_1^{et,nt}(X^o)$ depends on the compactification $X$ as well. Similarly $\pi_1^{et,ct}(X^o)$ 
denotes the quotient of $\pi_1^{et}(X^o)$ which corresponds to curve-tame covers of $X^o$.

\begin{pro}
The group $\pi_1^{et,nt}(X^o)$ is the inverse limit of $\pi_1^{et}(X,\,P)$, where the
limit is over all the branch data $P$ on $X$ whose branch locus is outside $X^o$ and $P$
is numerically tame. Similarly, $\pi_1^{et,ct}(X^o)$ is the limit over the branch
data $P$ which are curve-tame.
\end{pro}

\begin{proof}
Note that $f\,:\,Y\,\longrightarrow\, X$ is a numerically-tame covering if and
only if $B_f$ is a numerically-tame branch data. The same holds for
curve-tame \cite[Proposition 9.6]{formal.orbifolds}. Rest of the argument is
similar to the proof of \cite[Theorem 5.4]{formal.orbifolds}. 
\end{proof}

We recall from \cite{Nori-pi1} the definition of vector bundles on geometric formal orbifolds 
$(X,\,P)$ (also see \cite{KP} for the case of curves). Let $(Y,\,O)\,\longrightarrow\, (X,\,P)$ be 
an \'etale $\Gamma$--Galois covering of formal orbifolds where $O$ is the trivial branch data. The category $\Vect(X,\,P)$ of vector 
bundles on $(X,\,P)$ are the $\Gamma$--equivariant vector bundles on $Y$, while morphisms 
between two vector bundles on $(X,\,P)$ are defined to be the $\Gamma$--equivariant 
homomorphisms between the corresponding $\Gamma$--bundles on $Y$. It was observed on page 302 of \cite{Nori-pi1} that the category 
$\Vect(X,\,P)$ does not depend on the choice of the covering $(Y,\,O)\,\longrightarrow\, 
(X,\,P)$. In the case of curves, this is proved in \cite[Proposition 3.6]{KP}. 

Now we further assume that $X$ is also projective over $k$ and we fix a polarization of $X$. Let ${\Vect}^s(X,\,P)$ (respectively, ${\Vect}^f(X,\,P)$) denote the full subcategory 
of $\Vect(X,\,P)$ consisting of strongly semistable (respectively, essentially finite) 
$\Gamma$--equivariant vector bundles on $Y$. Here degree and slopes are computed with respect to the pullback to $Y$ of the fixed polarization on $X$ to define semistability. The category ${\Vect}^f(X,\,P)$ is 
equipped with the usual operations of direct sum, tensor product and dual. With these
${\Vect}^f(X,\,P)$ is a 
Tannakian category. For a closed point $x\,\in \,X$ outside the support of $P$, there is a fiber functor $E\, \mapsto \, E_{|x} $ from ${\Vect}^f(X,\,P)$ to the category of $k$--vector spaces. The 
automorphism group scheme of this fiber functor is the Nori fundamental group scheme 
$\pi^N((X,\,P),\,x)$ of $(X,\,P)$. Since Nori fundamental group scheme for varieties are profinite, it follows from \cite[Thoerem 3.1]{Nori-pi1} or \cite[Proposition 3.4]{Nori-pi1} that $\pi^N((X,\,P),x)$ is a profinite group scheme. We will often drop the base point from the notation.

We say that a cover $f\,:\,Y\,\longrightarrow\, X$ is genuinely ramified if there is no 
nontrivial \'etale cover of $X$ dominated by $f$.

\begin{lemma}\label{genuinelyramifiedhom}
Let $f\,:\,Y\,\longrightarrow\, X$ be a genuinely ramified map, and let $E_1,\, E_2$ be semistable bundles
on $X$ with $\mu(E_1)\,=\, \mu(E_2)$. Then $\Hom(f^*E_1,\,f^*E_2)\,\cong\, \Hom(E_1,\,E_2)$.
\end{lemma}

\begin{proof}
See \cite[p. 12844, Lemma 4.3]{BP} and the proof of Theorem 2.5 of \cite{BDP}.
\end{proof}

\begin{lemma} \label{subequivariantbundle-pullback}
Let $f\,:\,Y\,\longrightarrow\, X$ be a $G$-Galois cover, and let $E$ be a vector bundle on $X$.
Let $V$ be a $G$-equivariant subbundle of the $G$-equivariant vector bundle $f^*E$. Then there
is a natural isomorphism $V\,\cong\, f^*((f_*V)^G)$ of $G$-equivariant bundles.
\end{lemma}

\begin{proof}
We have $(f_*{\mathcal O}_Y)^G \,=\, {\mathcal O}_X$. Consequently, using the projection formula,
$$
(f_*f^* E)^G \,=\, ((f_*{\mathcal O}_Y)\otimes E)^G \,=\, E\, .
$$
Therefore,
$$
f^*((f_*V)^G) \, \subset\, f^*((f_*f^*E)^G)\,=\, f^*E,
$$
and $f^*((f_*V)^G)$ is a subbundle of $f^*E$. Now, the two subbundles $V$ and $f^*((f_*V)^G)$ of
$f^*E$ coincide over the open subset $f^{-1}(U)\, \subset\, Y$, where $U\, \subset\, X$ is the open
subset over which the map $f$ is \'etale. This implies that the two subbundles $V$ and $f^*((f_*V)^G)$ of
$f^*E$ coincide over the entire $Y$.
\end{proof}

\begin{pro}\label{embeddingVect}
Let $X$ be a normal projective variety. Let $P_1 \,\ge\, P_2$ be two geometric branch data on $X$.
Then there is a fully faithful functor $$i^*\,:\,{\Vect}^s(X,\,P_2)\,\longrightarrow\, {\Vect}^s(X,\,P_1).$$
Moreover, for any object $\scrE$ in ${\Vect}^s(X,\,P_2)$ and a subobject $\scrV$ of $i^*\scrE$,
there exist an object $\scrE'$ in ${\Vect}^s(X,\,P_2)$ such that $i^*\scrE'\,\cong\, \scrV$.
\end{pro}

\begin{proof}
For $i\,=\,1, \,2$, let $f_i\,:\,(Y_i,\,O)\,\longrightarrow\, (X,\,P_i)$ be an \'etale $G_i$-Galois cover. Replacing
$f_1$ by a dominating component of the fiber product of $f_1$ and $f_2$ we may 
assume $f_1$ dominates $f_2$.

As $f_1$ dominates $f_2$, there is a unique morphism $$f\,:\,Y_1\,\longrightarrow\, Y_2$$ such 
that $f_2 \circ f\,=\,f_1$. Let $H$ be the Galois group of the extension
$k(Y_1)/k(Y_2)$, so $G_1/H\,=\,G_2$. For a $G_2$--equivariant vector bundle $E$ on $Y_2$, note that $f^*E$ is a $G_1$--equivariant vector bundle (see \cite[Lemma 3.3]{KP}). The functor $i^*$ sends $E$ to $f^*E$.

Take $G_2$--equivariant vector bundles $E_1$ and $E_2$
on $Y_2$. We need to show that 
$$\Hom(E_1,\,E_2)^{G_2}\,\cong\,\Hom(f^*E_1,\,f^*E_2)^{G_1}.$$ For that it suffices to show that
$$\Hom(E_1,\,E_2) \,\cong\, \Hom(f^*E_1,\,f^*E_2)^H.$$

We may express $f$ as $f\,=\, h\circ g$, where $g\,:\,Y_1\,\longrightarrow\, Z$ is genuinely ramified
and $h\,:\,Z\,\longrightarrow\, Y_2$ is an \'etale $H'$--Galois cover. By Lemma 
\ref{genuinelyramifiedhom} we have $$\Hom(f^*E_1,\,f^*E_2)\,\cong\, \Hom(h^*E_1,\,h^*E_2).$$
By \'etale Galois descent,
$$\Hom(h^*E_1,\,h^*E_2)^{H'}\,\cong\,  \Hom(E_1,\,E_2).$$
Combining these we obtain that $\Hom(E_1,\,E_2) \,\cong\, \Hom(f^*E_1,\,f^*E_2)^H$.
 
For the second part of the proposition, let $E$ be a $G_2$-equivariant vector bundle on $Y_2$ 
so that $\scrE\,=\, (f_2\,:\,(Y_2,\,O)\,\longrightarrow\, (X,\,P_2),\,E)$,
and let $V$ be a $G_1$-equivariant subbundle of $f^*E$ representing $\scrV$.
Then $V$ is also an $H$-equivariant bundle. Hence, by Lemma
\ref{subequivariantbundle-pullback}, $$V\cong f^*[(f_*V)^H].$$ But
$(f_*V)^H$ is a $G_2$-equivariant subbundle of $E$.
\end{proof}

\begin{cor}
Let $P_1 \,\ge\, P_2$ be two geometric branch data on a normal projective variety $X$.
Then the natural morphism $\pi^N(X,\,P_1)\,\longrightarrow\, \pi^N(X,\,P_2)$ is surjective.
\end{cor}

\begin{proof}
In view of \cite[Proposition 2.21]{deligne-milne} and Proposition \ref{embeddingVect}, we only have to note that when $\scrE$ is in the subcategory $\Vect^f(X,\,P_2)$ then $i^*\scrE$ is in $\Vect^f(X,\,P_1)$. This is clear, since $i^*\scrE$ is given by the $G_1$--equivariant bundle $f^*E$ in the proof of Proposition \ref{embeddingVect} and $E$ essentially finite implies $f^*E$ is essentially finite.
\end{proof}

\begin{definition}
The \emph{numerically tame Nori fundamental group} $\pi^{N,nt}(X^o)$ of $X^o$ is
defined to be the projective limit of $\pi^N(X,\,P)$, where the limit is over all the geometric
branch data $P$ on $X$ whose branch locus is outside $X^o$ and $P$ is numerically tame.
Similarly, \emph{curve-tame Nori fundamental group} $\pi^{N,ct}(X^o)$ of $X^o$ is
the limit over the geometric branch data $P$ which are curve-tame.
\end{definition}

\begin{remark}
 Note that in \cite[Definition 3.3]{Nori-pi1} we introduced the notion of the group scheme $\pi^n(X^o)$ as being the projective limit of $\pi^N(X,\,P)$, where the limit is over all the branch data $P$ on $X$ whose branch locus is outside $X^o$. Hence, $\pi^{N,nt}(X^o)$ and $\pi^{N,ct}(X^o)$ are quotients of $\pi^n(X^o)$.
\end{remark}

\begin{definition}\label{dfnt}
Let $X$ be a smooth projective variety and $D$ a divisor on $X$. Let
$${\Vect}^{f,nt}(X,\,D)$$ be the category defined as follows:
\begin{itemize}
\item The objects are pairs $(f\,:\,Y\,\longrightarrow\, X,\,\, E)$, where $f$ is a numerically
tamely ramified Galois cover \'etale outside $D$, and $E$ is an $\Aut(Y/X)$-equivariant vector
bundle on $Y$, which is essentially finite.
 
\item Let $\scrE_1\,=\,(f_1\,:\,Y_1\,\longrightarrow\, X,\,\, E_1)$ and
$\scrE_2\,=\,(f_2\,:\,Y_2\,\longrightarrow\, X,\,\, E_2)$ be two objects, and let $f\,:\,Y\,\longrightarrow\,
X$ be a tamely ramified Galois cover of $X$ which is \'etale outside $D$ while
dominating both $f_1$ and $f_2$. Let $g_i\,:\,Y\,\longrightarrow\,Y_i$ be such that
$f_i \circ g_i\,=\,f$ for $i\,=\,1,\,2$. Then $g_1^*E_1$ and $g_2^*E_2$ are
$\Aut(Y/X)$-equivariant vector bundles over $Y$.

\item Denote by ${\rm Hom}(\scrE_1,\,\scrE_2)$ the space of $\Aut (Y/X)$-equivariant 
homomorphisms from $g_1^*E_1$ to $g_2^*E_2$. Note that by Proposition \ref{embeddingVect}, ${\rm Hom}(\scrE_1,\,
\scrE_2)$ does not depend on the choice of $Y$.

\item The tensor product is $$\scrE_1\otimes \scrE_2\,:=\,(f\,:\,Y\,\longrightarrow\, X,\,\, g_1^*E_1\otimes g_2^*E_2)$$ while the dual
is $\scrE_1^*\,=\,(f_1,\, E_1^*)$.
\end{itemize}
\end{definition}

\begin{pro}\label{2.9}
The category ${\Vect}^{f,nt}(X,\,D)$ in Definition \ref{dfnt} is Tannakian. Its Tannaka dual is
$\pi^{N,nt}(X\setminus D)$.
\end{pro}

\begin{proof}
Note that for geometric branch data $P$ and $P'$ on $X$, with  $P\,\le\, P'$, we have a
fully faithful functor from ${\Vect}^f(X,\,P)$ to ${\Vect}^f(X,\,P')$. Now the direct
limit of ${\Vect}^f(X,\,P)$, where $P$ varies over numerically tame geometric branch data
with branch locus $\BL(P)\,\subset\, D$, is precisely ${\Vect}^{f,nt}(X,\,D)$. Hence its
Tannaka dual is the inverse limit of the Tannaka duals of ${\Vect}^f(X,\,P)$, where $P$
varies over numerically tame geometric branch data with branch locus $\BL(P)\,\subset\, D$,
which is $\pi^{N,nt}(X\setminus D)$.
\end{proof}

Recall that a vector bundle $E$ on a variety $Y$ is called Frobenius trivial or $F$-trivial if $F^{n*}E$ is trivial for some $n$ where $F:Y\longrightarrow Y$ is the Frobenius morphism.

\begin{pro}\label{F-trivial}
The full subcategory of ${\Vect}^{f,nt}(X,D)$ consisting of objects $(f,\,E)$, where $E$ is
a $F$-trivial bundle, is equivalent to the entire category ${\Vect}^{f,nt}(X,D)$.
\end{pro}

\begin{proof}
Let $(f\,:\,Y\,\longrightarrow\, X,\,\,E)$ be an object of ${\Vect}^{f,nt}(X,\,D)$. Since $E$ is an essentially finite bundle on $Y$, there
exist a finite \'etale cover $g\,:\,Y'\,\longrightarrow\, Y$ such that $g^*E$ is $F$-trivial. Passing to the Galois
closure of $f\circ g$ we may assume $g$ is Galois and hence $g^*E$ is $\Aut(Y'/X)$-equivariant bundle. By
definition $(f\circ g\,:\,Y'\,\longrightarrow\, X,\,\, g^*E)$ is isomorphic to $(f,\,E)$ in ${\Vect}^{f,nt}(X,\,D)$. 
\end{proof}

In the same way we can define ${\Vect}^{f,ct}(X,\,D)$ and deduce that its Tannaka dual is $\pi^{N,ct}(X\setminus D)$. Moreover, Proposition \ref{F-trivial} holds with ${\Vect}^{f,nt}(X,\,D)$ replaced by ${\Vect}^{f,ct}(X,\,D)$.

\subsection{A definition using tame parabolic bundles}

Here we assume that $X$ is smooth, $k$ is an algebraically closed field of characteristic 
zero and $D\,=\,X\setminus X^o$ is a reduced effective divisor. In \cite{MY} 
parabolic bundles on curves $X$ relative to $D$ were defined. For higher dimensions this 
was extended in \cite{MY}.  In \cite{Bis} it was shown that when $k\,=\,\mathbb C$, and $D$
is a simple normal crossing divisor, the 
category of parabolic bundles is equivalent to the category of orbifold bundles (see
also \cite{Bo1}, \cite{Bo2}).

When $k$ is of positive characteristic, and $X$ is a curve, the orbifold bundles were 
studied in \cite{KP}. In that case, parabolic bundles were defined in \cite{KM} and it was 
shown that the two categories are equivalent. Though the structure of parabolic bundles 
is complicated due to the wild ramifications, the objects in the subcategory of tame 
parabolic bundles which correspond to the tame orbifold bundles are
considerably simpler. In fact, the
tame parabolic bundles are same as the parabolic bundles whose parabolic
weights have denominators prime to the characteristic $p$.

Let $X$ be a smooth projective variety and
$$
D\, \subset\, X
$$
a reduced effective divisor. Let
$$
D\, =\, \sum_{i=1}^\ell D_i
$$
be the decomposition of $D$ into its irreducible components.

Take a vector bundle $E$ on $X$. A quasiparabolic structure on $E$ is a filtration
of subbundles
\begin{equation}\label{e1}
E\big\vert_{D_i}\,=\, F^i_1\, \supset\, F^i_2 \, \supset\, \cdots \, \supset\,
F^i_{n_i-1} \, \supset\, F^i_{n_i} \, \supset\,F^i_{n_i+1}\,=\,0
\end{equation}
for every $1\, \leq\, i\, \leq\, \ell$. A system of parabolic weights for such a
quasiparabolic structure consists of rational numbers
\begin{equation}\label{e1a}
0\,\leq\, \alpha^i_1\, <\, \alpha^i_2\, <\, \cdots \, <\, \alpha^i_{n_i-1}\, <\, \alpha^i_{n_i}
\, < \, 1
\end{equation}
for every $1\, \leq\, i\, \leq\, \ell$ satisfying the condition that there is an integer
$N\, \geq\, 1$ prime to $p$ for which $N\alpha^i_j\, \in\, \mathbb Z$ for all $i,\,j$.

A parabolic structure on $E$ is a quasiparabolic structure as above together with a system of
parabolic weights. Take a parabolic structure on $E$. Let
$\{E_t\}_{t\in \mathbb R}$ be the corresponding filtration of sheaves (see \cite{MY}). We impose
the following condition on the parabolic structure:

Each sheaf $E_t$, $t\, \in\, \mathbb R$, is locally free.

A tame parabolic structure on $E$ is a parabolic structure on $E$ satisfying the above condition.

Tensor product and dual of parabolic bundles are defined in standard way (see, for 
example, \cite{Bi2}, \cite{Yo}). Let $E_*$ and $F_*$ be parabolic bundles. Then all the 
parabolic weights of the parabolic tensor product $E_*\otimes F_*$ at a parabolic divisor 
$D$ are of the form $\alpha+\beta - [\alpha+\beta]$, where $\alpha$ 
(respectively, $\beta$) is a parabolic weight of $E_*$ (respectively, $F_*$) at $D$; the 
notation $[t]\, \in\, {\mathbb Z}$ stands for the integral part of $t\, \in\, \mathbb 
R$, so $0\, \leq\, t-[t] \, <\,1$. From this it follows immediately that
the parabolic tensor product of two tame parabolic bundles is again a tame parabolic bundle.
Next note that all the parabolic weights of the parabolic dual $E^*_*$ at a parabolic divisor $D$ are 
of the form $-\alpha -[-\alpha]$, where $\alpha$ is a parabolic weight of $E_*$ at $D$. This implies
that the parabolic dual of a tame parabolic bundle is again a tame parabolic bundle.

For any polynomial $f\,=\, \sum_{i=0}^n a_i x^i \, \in\, {\mathbb Z}[X]$ with $a_i\, \geq\, 0$,
and any tame parabolic bundle $E_*$, define
$$
f(E_*)\,:=\, \bigoplus_{i=0}^n (E^{\otimes i}_*)^{\oplus a_i}\, ,
$$
where $E^{\otimes 0}_*$ is the trivial parabolic line bundle (trivial line bundle with no nonzero
parabolic weight).

A tame parabolic vector bundle $E_*$ is called \textit{finite} if there are two distinct
polynomials $f_1$ and $f_2$ as above such that the two parabolic vector bundle
$f_1(E_*)$ and $f_2(E_*)$ are isomorphic. If $E_*$ and $F_*$ are finite tame parabolic vector bundles,
then $E_*\otimes F_*$, $E_*\oplus F_*$ and $E^*_*$ are also finite tame parabolic vector bundles.

\begin{pro}\label{2.11a}
Any finite tame parabolic vector bundle $E_*$ is semistable, and the parabolic degree of
$E_*$ is zero.
\end{pro}

\begin{proof}
A tame parabolic vector bundle $E_*$ is finite if and only if there are finitely many tame parabolic
vector bundles $V_{1,*},\, \cdots,\, V_{m,*}$ such that for every $j\, \geq\, 1$, we have
$$
E^{\otimes j}_*\ =\ \bigoplus_{i=1}^m (V_{i,*})^{\oplus n^j_i}
$$
for nonnegative integers $n^j_i$, where the tensor product and direct sum are in the parabolic category;
see \cite[p.~80, Lemma 3.1]{No} for a proof. Using this it is straightforward to deduce the proposition;
see \cite[p.~6, Lemma 3.3]{Bi3}.
\end{proof}

{}From \ref{2.11a} it follows that the abelian category generated by the finite tame parabolic vector
bundles consists of tame parabolic vector bundles (note that any parabolic
subbundle of a tame parabolic bundle is tame).

An \textit{essentially finite} tame parabolic bundle is a parabolic bundle lying in the abelian category
generated by the finite tame parabolic vector bundles.

\begin{pro}\label{2.11}
The category of essentially finite tame parabolic vector bundles with respect to $D$ form a
Tannakian category. 
\end{pro}

\begin{proof}
As noted above, the essentially finite tame parabolic vector bundles with respect to $D$ form an
abelian category. Next we note that this is a tensor category. Let $E_*$ and $E'_*$ be two semistable
parabolic vector bundles of parabolic degree zero on a complete curve. If a parabolic homomorphism 
$\varphi\, :\, E_*\, \longrightarrow\, E'_*$ vanishes at a point, then it can be shown that $\varphi\,=\, 0$.
Indeed, this follows by considering the image of $\varphi$. Its parabolic degree is zero because it is
simultaneously a quotient of $E_*$ and a subsheaf of $E'_*$. On the other hand, if $\varphi$ is not identically
zero but vanishes at a point, then it generates a parabolic subsheaf of $E'_*$ of strictly larger parabolic
degree (than itself). But this contradicts the given condition that $E'_*$ is parabolic semistable of parabolic
degree zero. Using this it follows that the essentially finite tame parabolic vector bundles with respect
to $D$ form a rigid tensor category.
See the proof in \cite{No} that the essentially finite vector bundles form a
Tannakian category for very similar arguments.
\end{proof}

Let $\pi^{N,pt}(X^o)$ denote the Tannakian dual of this category of essentially finite 
tame parabolic bundles with respect to $D$. This $\pi^{N,pt}(X^o)$ is the third candidate 
for the Nori fundamental group scheme.

A tame orbifold bundle on $X$ relative to $D$ is a vector bundle on a formal orbifold 
$(X,\,P)$ for some tame branch data $P$ with $\BL(P)$ contained in $D$. When $D$ is a strict 
normal crossing divisor, extending the result of \cite{Bis} we will show that the category 
of tame orbifold bundles on $X$ relative to $D$ is equivalent to the category of tame 
parabolic bundles on $X$ relative to $D$.

\subsection{A definition using linearly reductive group scheme torsors}

One definition which is attributed to Borne and Vistoli (Definition 10.4 of \cite{BV}) is the maximal pro linearly
reductive quotient of $\pi^N(X^o)$ using the tame Nori fundamental gerbe (also see Definition 3.9 and
Remark 3.10 of \cite{Otabe}). Let us denote this
group scheme by $\pi^{tame}(X^o)$. It behaves well with respect to the
extension of base field. It surjects onto the prime-to-$p$ part of the \'etale fundamental group of $X^o$.

It is natural to pose the following question.

\begin{question}
What is the relation between $\pi^{tame}(X^o)$ and $\pi^{N,nt}(X^o)$?
\end{question}

\section{Relationship between $\pi^{N,nt}(X^o)$, $\pi^{N,ct}(X^o)$ and $\pi^{N,pt}(X^o)$}

Let us first assume $X^o$ is a curve. Denote by $X$ its smooth compactification, and set 
$D\,=\,X\setminus X^o$. In this case $\pi^{N,nt}(X^o)$ and $\pi^{N,ct}(X^o)$ are evidently the 
same. There is an equivalence of categories between the category of tame orbifold bundles on 
$X$ with orbifold structure along $D$ and the category of tame parabolic bundles on $X$ with 
filtration along $D$ (see \cite[Proposition 5.15]{KM}). This induces an equivalence between the 
category ${\Vect}^{f}(X,\,D)$ of essentially finite tame orbifold bundles on $X$ with orbifold 
structure along $D$ and the category $\PVect^{f}(X,\,D)$ of essentially finite tame parabolic bundles on $X$ with
filtration along $D$. Consequently, this equivalence of categories
gives an isomorphism between $\pi^{N,nt}(X^o)$ and 
$\pi^{N,pt}(X^o)$, because the former is the Tannaka dual of ${\Vect}^{f}(X,\,D)$ and the 
latter is the Tannaka dual of $\PVect^{f}(X,\,D)$.

\begin{pro}\label{3.1}
Let $X$ be a smooth projective variety over $k$ and $D\, \subset\, X$ a reduced divisor
on $X$. Let $X^o\,:=\,X\setminus D$ be the complement of it. Then there is an epimorphism
$\pi^{N,ct}(X^o)\,\longrightarrow\, \pi^{N,nt}(X^o)$.
\end{pro}

\begin{proof}
Note that if $Y\,\longrightarrow\, X$ is a covering, \'etale over $X^o$, such that the 
ramification along the divisor $D\,=\,X\setminus X^o$ is numerically tame, then the 
pull-back of this covering via a nonconstant morphism $C\,\longrightarrow\, X$, where $C$ 
is smooth connected projective curve, is tamely ramified \cite[Theorem 5.3 and Theorem 
5.4]{KS-tame}. This implies that if $P$ is geometric numerically tame branch data, then it 
is curve-tame as well \cite[Proposition 9.7]{formal.orbifolds}. Hence there is a 
surjection $\pi^{N,ct}(X^o)\,\longrightarrow\, \pi^{N,nt}(X^o)$.
\end{proof}

In higher dimensions when $D$ is a strict normal crossing divisor, we first note that 
$\pi^{N,ct}(X^o)\,\longrightarrow\, \pi^{N,nt}(X^o)$ is an isomorphism. This is because a 
curve-tame cover branched along a strict normal crossing divisor is numerically tame (\cite[Theorem 1.1 and 1.2]{KS-tame}).  Note that in the appendix of \cite{KS-tame} it is also shown that curve-tame covers need not be numerically tame in general. Hence the map $\pi^{N,ct}(X^o)\,\longrightarrow\, \pi^{N,nt}(X^o)$ is not necessarily an isomorphism if $D$ is not a strict normal crossing divisor.

We will show that $\pi^{N,nt}(X^o)$ and $\pi^{N,pt}(X^o)$ are isomorphic. 
It actually follows from Theorem \ref{orb-par-corr}.

\begin{thm}\label{orb-par-corr}
Assume that $D$ is a strict normal crossing divisor.
The category of tame parabolic bundles on $X$ relative to $D$ is equivalent
to the category of tame orbifold bundles on $X$ relative to $D$. Moreover, this equivalence holds even when restricted to the subcategory of essentially finite objects in the two categories.
\end{thm}

\begin{proof}
In \cite{Bis}, \cite{Bo1}, \cite{Bo2} this is proved under the assumption that the
characteristic of the base field $k$ is zero. However, when the characteristic of
$k$ is positive, the same proofs works if we restrict to the tame parabolic bundles.

The moreover part follows from the fact that the above orbifold parabolic correspondence commutes with tensor products and direct sums.
\end{proof}

\begin{question}
In general when $D$ may not be normal crossing divisor, is $\pi^{N,ct}(X^o)$
isomorphic to $\pi^{N,pt}(X^o)$?
\end{question}

\section{Lefschetz theorem for tame Nori fundamental group}

Let $X$ be an irreducible smooth projective variety and $D\, \subset\, X$ a simple normal crossing divisor.
Let
\begin{equation}\label{eqio}
\iota\ :\ X_0\ :=\ X\setminus D \ \hookrightarrow\ X
\end{equation}
be the inclusion map.

Take a parabolic vector bundle $E_*$ on $X$ with parabolic structure over $D$. The vector bundle
$E_0$ underlying $E_*$ will be denoted by $E$, for notational convenience. Consider the parabolic
tensor power $E^{\otimes p}_*$, where $p\, >\, 0$ is the characteristic of the base field $k$; for
notational convenience, $E^{\otimes p}_*$ will be denoted by $\mathbb{E}_*$. Note that
$\iota^* \mathbb{E}_t\,=\, \iota^* E^{\otimes p}$ for all $t\, \in\, \mathbb R$, where $\iota$ is the map
in \eqref{eqio}. In particular, $\mathbb{E}_t$ is a subsheaf of the quasi-coherent sheaf
$\iota_* \iota^* E^{\otimes p}$. The two vector bundles $\mathbb{E}_t$ and $(E_t)^{\otimes p}$ do not coincide
in general; the notation ${\mathbb E}_*$ was introduced in order to make this distinction.

We have a subbundle
\begin{equation}\label{ece}
\widetilde{\mathcal E}\ :=\ \{v^{\otimes p}\, \in\, \iota^* E^{\otimes p}\,\, \big\vert\,\, v\, \in\, \iota^* E\}
\ \subset \ \iota^* E^{\otimes p}
\end{equation}
over $X_0$. Note that the map $\iota^* E\, \longrightarrow\, \iota^* E^{\otimes p}$, $v\, \longmapsto\,
v^{\otimes p}$, is not a morphism of vector bundles because it is not ${\mathcal O}_{X_0}$--linear. However,
the image of this map is a subbundle of $\iota^* E^{\otimes p}$, which we have denoted by $\widetilde{\mathcal E}$.

As mentioned above, $\mathbb{E}_t$ is a subsheaf of $\iota_* \iota^* E^{\otimes p}$. Now define
\begin{equation}\label{fppb}
{\mathcal E}_{t}\ :=\ \mathbb{E}_t\cap \iota_* \widetilde{\mathcal E}\ \subset \
\iota_*\iota^* E^{\otimes p};
\end{equation}
note that we have $\iota_* \widetilde{\mathcal E}\, \subset\, \iota_*\iota^* E^{\otimes p}$ (see \eqref{ece}).
The intersection of sheaves in \eqref{fppb} takes place inside $\iota_*\iota^* E^{\otimes p}$.

Since $\mathbb{E}_{a+b}\, \subset\, \mathbb{E}_b$ for all $a\, \geq\, 0$, it follows
that ${\mathcal E}_{s+t}\, \subset\, {\mathcal E}_t$ for all $s\, \geq\, 0$. We have
$\mathbb{E}_{t+1}\,=\, \mathbb{E}_t\otimes {\mathcal O}_X(- D)$ for all $t$, and also
$\iota_*\iota^* E^{\otimes p}\,=\, \iota_*\iota^* (E^{\otimes p}\otimes {\mathcal O}_X(-D))$.
Using these, from \eqref{fppb}
it follows immediately that ${\mathcal E}_{t+1}\,=\, {\mathcal E}_t
\otimes {\mathcal O}_X(-D)$ for all $t$.
Also, if $\mathbb{E}_{s+t}\,=\, \mathbb{E}_{t}$ for some $s\, \in\, {\mathbb R}$ and some $t\, >\, 0$, then clearly we have
${\mathcal E}_{s+t}\,=\, {\mathcal E}_{t}$. All these together imply that $\{{\mathcal E}_t\}_{t\in {\mathbb R}}$ defines a parabolic
vector bundle on $X$ with parabolic structure over $D$. Let
\begin{equation}\label{fppb2}
{\mathcal E}_*\ :=\ \{{\mathcal E}_t\}_{t\in\mathbb R}
\end{equation}
denote this parabolic vector bundle.
\begin{definition}
The parabolic bundle ${\mathcal E}_*$ is called the Frobenius pullback of the parabolic bundle $E_*$.
\end{definition}

Note that for usual vector bundles --- meaning when the parabolic divisor $D$ is the zero
divisor --- we have ${\mathcal E}_t\,=\, {\mathcal E}_0\,=\, F^*_X E$, where
$$
F_X\ :\, X\ \longrightarrow\ X
$$
is the absolute Frobenius morphism of $X$.

Let $({\mathcal E}_*)^*$ be the parabolic dual of the parabolic vector bundle ${\mathcal E}_*$ in \eqref{fppb2}.
Then $({\mathcal E}_*)^*$ coincides with the parabolic vector bundle constructed as in \eqref{fppb2}
by substituting the parabolic dual $E^*_*$ in place of $E_*$.
Take parabolic vector bundles $E^1_*$ and $E^2_*$. Let ${\mathcal E}^1_*$ and ${\mathcal E}^2_*$ be the
parabolic vector bundles constructed as in \eqref{fppb2} by substituting $E^1_*$ and $E^2_*$ respectively
in place of $E_*$. The parabolic vector bundles $E^1_*\oplus E^2_*$ and $E^1_*\otimes E^2_*$ will be denoted by
$V_*$ and $W_*$ respectively. Let ${\mathcal V}_*$ and ${\mathcal W}_*$ be the
parabolic vector bundles constructed as in \eqref{fppb2} by substituting $V_*$ and $W_*$ respectively in place of $E_*$.
Then ${\mathcal V}_*$ (respectively, ${\mathcal W}_*$) is identified with ${\mathcal E}^1_*\oplus
{\mathcal E}^2_*$ (respectively, ${\mathcal E}^1_*\otimes{\mathcal E}^2_*$). In other words we have the following.

\begin{pro}
 The Frobenius pullback of parabolic bundles commutes with taking parabolic dual, direct sums and tensor products.
\end{pro}

Consider the vector bundle $\widetilde{\mathcal E}\ \longrightarrow\ X_0$ (see \eqref{eqio} and \eqref{ece}).
It has a canonical integrable connection whose integrable sections over any open subset $U'\, \subset\, X_0$
are of form $s^{\otimes p}$, where $s$ is any section of $E$ over $U'$. This produces an integrable connection
on the Frobenius pullback of vector bundles (usual vector bundles --- not parabolic bundles).

Take a logarithmic connection $\nabla$ on a vector bundle $W$ singular over a divisor $\Delta$. Perform an elementary
transformation of $W$ over $\Delta$ using a subbundle $V\, \subset\, W\big\vert_\Delta$. Let $W'$
be this vector bundle obtained using elementary transformation of $W$.
Then $\nabla$ produces a logarithmic connection on $W'$ if and only if $V$ is preserved by the residue of $\nabla$.

Since the vector bundles $\mathcal{E}_t$, $t\, \in\, \mathbb R$ (see \eqref{fppb2}), are obtained using elementary
transformations starting from the Frobenius pullback of the vector bundle $\mathcal{E}_0$, the above fact implies
that the integrable connection on the Frobenius pullback of the vector bundle $\mathcal{E}_0$ produces a
logarithmic connection on $\mathcal{E}_t$ for all $t\, \in\, {\mathbb R}$,
where ${\mathcal E}_t$ is the vector bundle in \eqref{fppb2}. This
logarithmic connection on $\mathcal{E}_t$ will be denoted by
\begin{equation}\label{ae1}
\nabla^{\mathcal{E}_t}.
\end{equation}

When the parabolic divisor $D$ is zero, the above logarithmic connection $\nabla^{{\mathcal
E}_t}$ on $\mathcal{E}_t$ coincides  with the canonical integrable connection on the
Frobenius pullback $F^*_X E$ (recall that ${\mathcal E}_t\,=\, F^*_X E$ when $D$ is zero).

For any parabolic subbundle $V_*$ of the parabolic vector bundle $E_*$, the corresponding parabolic vector
bundle $\mathcal{V}_*$ --- constructed as in \eqref{fppb2} by substituting $V_*$ in place of $E_*$ --- is
clearly a parabolic subbundle of the parabolic vector bundle $\mathcal{E}_*$. The following converse
of it is a straightforward consequence of the construction of the logarithmic connection in \eqref{ae1}.

\begin{lemma}\label{alem1}
Let $W_*$ be a parabolic subbundle of the parabolic vector bundle $\mathcal{E}_*$.
There is a parabolic subbundle $V_*\, \subset\, E_*$ such that the corresponding
parabolic vector bundle $\mathcal{V}_*$ coincides with $W_*$ if and only if $W_t$ is
preserved by the logarithmic connection $\nabla^{\mathcal{E}_t}$ on $\mathcal{E}_t$
(see \eqref{ae1}) for all $t\, \in\, \mathbb R$.
\end{lemma}

Take a parabolic vector bundles $E_*$ on $X$.
Let $f\, :\, Y\, \longrightarrow\, X$ be a tamely ramified Galois covering satisfying the
following condition: There is an equivariant vector bundle $\widehat{E}$
on $Y$ which corresponds to ${E}_*$. Then the parabolic vector bundle $\mathcal{E}_*$
corresponds to the equivariant vector bundle $F^*_Y \widehat{E}$, where $F_Y\, :\, Y\,
\longrightarrow\, Y$ is the Frobenius morphism of $Y$.

Let $Y$ be an irreducible smooth projective variety over $k$. Consider the category
$\text{Vect}^f(Y)$ of essentially finite vector bundles on $Y$. For a reduced smooth very ample
hypersurface $\iota_H\, :\, H\, \hookrightarrow\, Y$, let $\text{Vect}^f(H)$ denote the category
of essentially finite vector bundles on $H$. Then the restriction functor
$$
\text{Vect}^f(Y) \ \longrightarrow\ \text{Vect}^f(H), \ \ \ V \, \longmapsto\, \iota^*_HV
$$
is fully faithful, provided $H$ is of sufficiently large degree (larger than an effective number
independent of $V$) (\cite[Theorem 1.1 (1)]{Biswas-Holla}, \cite[Theorem 10.2]{Langer}).

Also, for any essentially finite vector bundle $V$ on $Y$, any essentially finite subbundle $W$ of
$\iota^*_HV$ is of the form $\iota^*_H W'$ for an essentially finite vector bundle $W'$ on $Y$,
provided the degree of $H$ is sufficiently large (larger than an effective number
independent of $V$ and $W$) (\cite[Theorem 1.1 (1)]{Biswas-Holla},\cite[Theorem 10.2]{Langer}).

We observe below that the same holds for essentially finite parabolic vector bundles on $(X,\, D)$. Fix a very ample line bundle ${\mathcal O}_X(1)$ on $X$. Now take a reduced smooth
hypersurface $$\iota_{Z} \, :\, {Z}\, \longrightarrow\, X$$ with
\begin{equation}\label{ed}
{Z}\, \in\, \big\vert {\mathcal O}_X(d)\big\vert\,=\,
\big\vert {\mathcal O}_X(1)^{\otimes d}\big\vert,
\end{equation}
such that $Z$
intersects the parabolic divisor $D$ (see \eqref{eqio}) transversely. Let
$$
\text{PVect}^f(X,\, D)
$$
denote the category of essentially finite tame parabolic vector bundles on $X$ with parabolic
structure on $D$. Let
$$
\text{PVect}^f({Z},\, {Z}\cap D)
$$
denote the category of essentially finite tame parabolic vector bundles on $Z$ with parabolic
structure on ${Z}\cap D$. Note that the morphisms in this category are bundle homomorphisms which preserve the filtration.
\begin{pro}
 The restriction functor
$$
\text{PVect}^f(X,\, D)\ \longrightarrow\ \text{PVect}^f({Z},\, {Z}\cap D),\ \
\ W_* \, \longmapsto\, \iota^*_{Z}W_*
$$
is fully faithful, provided the degree of $Z$ is sufficiently large.
\end{pro}

The proof is same as in \cite[Theorem 10.2]{Langer} with the slight modification. One replaces the sheaf of homomorphisms $\mathcal{H}om(E,E')$ between two vector bundles $E$ and $E'$  by  the sheaf of filtration preserving homomorphisms $\mathcal{H}om({E}_*,{E'}_*)$ between two parabolic vector bundles ${E}_*$ and ${E'}_*$. Also in the proof of \cite[Theorem 9.1]{Langer}, one observes that the $\mathcal{H}om(E_*,{E'}_*)\cong (E_*)^*\otimes {E'}_*$ is an essentially finite tame parabolic bundle and its iterated Frobenius pullbacks are also essentially finite tame parabolic bundle. Hence by Theorem \ref{orb-par-corr} and the fact that the family of strongly semistable equivariant vector bundles of fixed rank and vanishing Chern classes is bounded, we get that the Hom sheaf and its iterated Frobenius pullbacks lie in a bounded family.

Hence using \cite[Lemma 2.1]{BiswasTohuku} (which is an improvement of \cite[Proposition 2.21]{deligne-milne} for profinite group schemes) we obtain the final result.
 \begin{thm}\label{4.5}
  The natural map
 $$\pi^{N.pt}(Z\setminus Z\cap D)\,\longrightarrow\, \pi^{N.pt}(X\setminus D)$$
 is surjective (faithfully flat) if the dimension of $Z$ is at least one.
 \end{thm}

\section*{Acknowledgements}

We thank the referee for helpful comments to improve the exposition. The first
author is partially supported by a J. C. Bose Fellowship (JBR/2023/000003). The second author is partially supported by Core grant (File no: CRG/2023/006248).

\end{document}